\documentclass[11pt]{amsart}

\usepackage{epsfig}
\usepackage{graphicx}
\usepackage{amscd}
\usepackage{amsmath}
\usepackage{amsfonts}
\usepackage{mathrsfs}
\usepackage{amssymb}
\usepackage{verbatim}
\usepackage{xcolor}
\usepackage{hyperref}

\RequirePackage{mathtools, amsthm, graphicx, mathrsfs, dsfont}

\newtheorem{theorem}{Theorem}[section]
\theoremstyle{plain}

\newtheorem{claim}{Claim}

\newtheorem{corollary}[theorem]{Corollary}

\newtheorem{lemma}[theorem]{Lemma}

\newtheorem{prop}[theorem]{Proposition}
\newtheorem{remark}[theorem]{Remark}

\numberwithin{equation}{section}

\def\beps{\boldsymbol{\eps}}

\def\Xxi{{X^{\vec s}_\zeta}}

\def\One{{1\!\!1}}

\def\dist{{\rm dist}}

\def\Lip{{\rm Lip}}

\def\Proj{{\rm Proj}}

\def\half{\frac{1}{2}}

\newcommand{\lam}{\lambda}

\newcommand{\gam}{\gamma}
\newcommand{\om}{\omega}

\newcommand{\Ups}{\Upsilon}
\newcommand{\Gam}{\Gamma}
\newcommand{\sig}{\sigma}

\def\bbe{{\bf e}}

\newcommand{\R}{{\mathbb R}}
\newcommand{\Q}{{\mathbb Q}}
\newcommand{\Z}{{\mathbb Z}}
\newcommand{\C}{{\mathbb C}}

\def\N{{\mathbb N}}

\def\Pr{{\rm Pr}}

\def\bp{{\bf p}}

\def\wt{\widetilde}
\def\A{{\mathcal A}}

\def\Lk{{\mathcal L}}

\def\Sf{{\sf S}}

\def\bx{{\mathbf x}}

\def\bz{{\mathbf z}}
\def\one{\vec{1}}
\def\be{\begin{equation}}
\def\ee{\end{equation}}
\newcommand{\Ek}{{\mathcal E}}

\newcommand{\eps}{{\varepsilon}}

\def\ov{\overline}

\newcommand{\const}{{\rm const}}

\def\ve1{\vec{1}}

\def\Ak{{\mathcal A}}

\def\Cyl{{\rm Cyl}}
\def\re{{\rm Re}}
\def\im{{\rm Im}}
\def\Tr{{\rm Tr}}

\def\proj{{\rm Proj}}

\def\beq{\begin{equation}}
\def\eeq{\end{equation}}

\def\wh{\widehat}
\def\lam{\lambda}
\def\nula{\nu_\lam}
\def\nulahat{\wh{\nu}_\lam}

\def\shalf{{\textstyle{\half}}}

\newcommand{\floor}[1]{\lfloor#1\rfloor}

\providecommand{\norm}[1]{\lvert\lvert #1 \rvert\rvert}

\RequirePackage{mathtools, amsthm, graphicx, mathrsfs, dsfont}

\makeatletter
\@namedef{subjclassname@2020}{\textup{2020} Mathematics Subject Classification}
\makeatother

\begin{document}
	
	\title[Q.W.M. for Salem substitutions]{Quantitative weak mixing for typical Salem substitution suspension flows}

	\author{Juan Marshall-Maldonado}
	\address{Juan Marshall-Maldonado\\ Faculty of Mathematics and Computer Science, Nicolaus
		Copernicus University, ul. Chopina 12/18, 87-100, Toru\'n, Poland}
	\email{jgmarshall21@gmail.com}
	
	\author{Boris Solomyak}
	\address{Boris Solomyak\\ Department of Mathematics,
		Bar-Ilan University, Ramat-Gan, Israel}
	\email{bsolom3@gmail.com}
	
	\thanks{The research of J.M. was supported by UMK grant (Poland) for young researchers from abroad 2024/2025, Excellence Center:
``Dynamics, Mathematical Analysis and Artificial Intelligence''. }
	\thanks{The research of B.S. was supported by the Israel Science Foundation grant \#1647/23.}
	
	\begin{abstract}
		The paper investigates quantitative weak mixing of Salem substitutions flows. We prove that for a substitution whose substitution matrix is 
		irreducible over the rationals and the dominant eigenvalue is a Salem number, for almost every suspension flow with a piecewise constant roof function,
		quantitative weak mixing holds with a rate that is slightly worse than a
		power of $\log\log$. We do not know if this is sharp, but we do show that
		 for any suspension flow of this kind, quantitative weak mixing with a polynomial rate is impossible. Results for specific systems are often much weaker than for
		 ``typical'' or ``generic'' ones. In the Appendix we explain how a minor modification of an argument from Bufetov and Solomyak (2014) yields very weak, but
		 nevertheless quantitative weak mixing estimates of $\log^*$ type for the {\em self-similar} suspension flow over a Salem substitution. 
		 Simultaneously this provides first quantitative decay rates for the Fourier transform of Salem Bernoulli convolutions.
	\end{abstract}
	
	\date{\today}


	\maketitle

	\thispagestyle{empty}
	
	\section{Introduction}
	
	In this paper we study ergodic-theoretic and spectral properties of substitution dynamical systems, more specifically, suspension flows over Salem type
	substitutions. Substitutions and associated dynamical systems appear in many different contexts, in particular, in the study of interval exchange 
	transformations with a periodic Rauzy induction, translation flows along stable/unstable foliations of pseudo-Anosov diffeormorphisms, see \cite{CGM}.
	as well as in 
	mathematics of quasicrystals and ``aperiodic order,'' see \cite{BaGr}.
	
	The basic definitions on substitutions may be found in the standard references \cite{Siegel,Queff}. Consider an alphabet with $d \geq 2$ symbols, denoted by $\A = \{1, \dots, d\}$. Let $\A^+$ represent the set of all nonempty words composed of letters from $\A$.  
	A {\em substitution} is a mapping $\zeta: \A \to \A^+$, which extends naturally to $\A^+$ and $\A^{\Z}$ by concatenation. The {\em substitution space} consists of all bi-infinite sequences $x \in \A^\Z$ such that every word appearing in $x$ is a subword of $\zeta^n(a)$, for some $a \in \A$ and some $n \in \N$. The {\em substitution dynamical system} is given by the left shift on $\A^\Z$, restricted to $X_\zeta$, and is denoted by $T$.  
	
	The {\em substitution matrix} $\Sf_\zeta = (\Sf(i,j))$ is the $d \times d$ matrix where $\Sf(i,j)$ represents the number of times the symbol $i$ appears in $\zeta(j)$. The substitution is called {\em primitive} if there exists some $n \in \N$ such that all entries of $\Sf_\zeta^n$ are strictly positive.  
	It is well known that the $\Z$-action given by the left-shift on a primitive substitution subshift is minimal and uniquely ergodic; see \cite{Queff}. We assume that the substitution is both primitive and {\em non-periodic}, which, in the primitive case, is equivalent to the space $X_\zeta$ being infinite. 
	
	The continuous counterpart comes from considering suspensions over $X_\zeta$.
	Namely, we consider $\R$-actions obtained as
	suspension (or special) flows over $(X_\zeta,T)$, with a piecewise-constant roof function depending only on the 1st symbol. Such a function
	may be identified with a positive {\em suspension vector} $\vec s\in \R^d_+$, see the next section for details.
	Substitution $\R$-actions of this kind appear in the study of translation flows and as tiling dynamical systems on the line.
	Special choices for the suspension vector are $\vec{s} = {\bf 1}$ (the vector of $1$'s), and $\vec{s} =$ the right PF (Perron--Frobenius) eigenvector of $\Sf_\zeta^{\sf T}$. With the first choice, we essentially recover the substitution subshift. The second choice is referred to as the {\em self-similar} suspension flow.
	
	By a result of Dekking--Keane \cite{DK} and Clark--Sadun \cite{CS03}, substitution systems and their suspensions are never strongly
	mixing, however, they may be weakly mixing, see \cite{FMN} for an algebraic characterization. Here our focus is on the property of
	{\em quantitative weak mixing} (QWM), which has been intensively studied recently.
	
	Weak mixing is a fundamental property of measure-preserving systems; it has several equivalent characterizations, see, e.g., \cite{EW}. One of them is absence of non-trivial discrete spectrum; equivalently, the spectral measure $\sig_f$ corresponding to any $L^2$-function of mean zero has no point masses.
Given a weakly mixing system, it is natural to ask what is the rate of decay of the spectral measure of a small neighborhood of a point, as the radius tends to zero, for a 
sufficiently ``nice'' test function $f$. Another 
characterization of weak mixing is in terms of absolute Ces\`aro averages: a probability preserving flow $(X, \phi_t, \mu)$ is weakly mixing if and only if
$$
\lim_{R\to \infty} \frac{1}{R} \int_0^R \left|\langle f\circ \phi_t, g\rangle \right|^2 dt \to 0,\ \ \mbox{for all}\ f\in L^2_0(X), \ g\in L^2(X).
$$
The flow is said to have the property of  
QWM with a rate $h$ for some increasing function $h:[0,+\infty)\to [0,+\infty)$, $h(0)=0$, if for all $f\in L^2_0(X)$ sufficiently ``nice'' and
$g\in L^2(X)$,
\beq \label{def:QWM}
\frac{1}{R}\int_0^R \left|\langle f\circ \phi_t, g\rangle \right|^2 dt \le C_{f,g}\cdot h(R^{-1}),
\eeq
see \cite{Knill98}.
If this holds with $h(r) = r^\gam$, we say that QWM holds with a H\"older rate $\gam$, for $h(r) = (\log(1/r))^{-\gam}$ we say it is $\log$-H\"older, etc.

\begin{remark} \label{rem:QWM}
It is known that spectral bounds of the form
\beq \label{spec-bound}
\sig_f(B_r(\om))\le C_f h(r), \ \ \om\in \R,\ r\in (0, r_0),\ \ \int f\,d\mu=0,
\eeq
imply QWM with the same rate, for the same $f$. (It is essential that \eqref{spec-bound} be {\bf uniform in the spectral parameter} $\om$.)
This was shown by Strichartz \cite{Stri90} for $h(r) = r^\alpha$, and extended by Last \cite{Last96}; see
Knill \cite{Knill98}. In the reverse direction, it was verified in \cite{Last96,Knill98} that QWM with the rate $h^2$ implies \eqref{spec-bound}.
\end{remark}

	
Let us return to the primitive substitution $\zeta$.	
	The Perron-Frobenius (PF) eigenvalue of $\Sf_\zeta$,
	which we denote by $\alpha$, plays an important role. Assume, for simplicity, that $\Sf_\zeta$ is irreducible over $\Q$. If $\alpha$ is a {\em Pisot number},
	which means that all the other eigenvalues of $\Sf_\zeta$ are inside the open unit disk, then the substitution $\Z$-action and suspension flows over it have a
	large discrete component, hence not weakly mixing, see \cite{CS03} (conjecturally, they are {\em pure discrete}, see \cite{ABBLS}).
	The case when the second-largest eigenvalue satisfies $|\alpha_2|>1$ was studied in \cite{BuSo14,BuSo18b}: 
	under this assumption, QWM with a H\"older rate was obtained for Lebesgue almost every suspension vector $\vec s$. 
	(There it was assumed that the test function is a ``cylindrical function of level zero,'' see the next section for definitions, which is, in a sense, the ``nicest'' function possible. In fact, this can be upgraded to the case of ``weakly Lipschitz'' test functions, see \cite{BuSo21}.) In \cite{BuSo18} and \cite{BuSo20} analogous results were
	obtained for almost every translation flow on a flat surface of genus 2. Forni \cite{Forni22}, 
	using a different method, obtained H\"older QWM for translation flows for an arbitrary genus $\ge 2$. After seeing his preprint, Bufetov and the second-named
	author obtained in \cite{BuSo21} similar results, using a ``symbolic approach''. More recently, Avila, Forni, and Safaee \cite{AFS} established H\"older QWM
	for almost all interval exchange transformations (IET's) of non-rotation type. Significantly,
	 all of these results on QWM rely on the second Lyapunov exponent being positive, either
	for the Kontsevich-Zorich, or for the Rauzy-Zorich cocycle --- the property which corresponds to $|\alpha_2|>1$ in the substitution case.
	
	There are some isolated results which do not fall into the above scheme: Moll \cite{Moll23} obtained power-log QWM for a substitution conjugate to the 
	Chacon map, for which $\Sf_\zeta$ has an eigenvalue 1; in the IET setting this is related to another result in \cite{AFS}, namely, power-log QWM for
	almost all rotation class IET's that are not rotations.
	
	In this paper we obtain first quantitative estimates for suspension flows over {\em Salem type} substitutions, that is, when 
	$\alpha$ is a Salem number. Recall
	that this means $\alpha$ is an algebraic integer, whose conjugates satisfy $|\alpha_j|\le 1$, with at least one of them on the unit circle.
	This is a  ``borderline'' case, interesting from several points of view and generally more difficult to analyze. Our main result, which we state
	precisely in the next section, asserts that for Lebesgue a.e.\ suspension flow over a Salem type substitution, holds QWM with a rate that is a
	power of $\log\log$, up to a $\log\log\log$ correction, namely,  with a function
	\beq\label{def:hrate}
	h_\beta(r)  := \exp\left(-\frac{\beta \log\log\log_\alpha (1/2r)}{\log\log\log \log_\alpha (1/2r)}\right).
	\eeq
	Here ``$\log$'' denotes the logarithm to base 2. On the other hand, we show that for {\em any} suspension vector $\vec s$, the resulting $\R$-action cannot
	have H\"older QWM.
	
	As a rule, it is much harder to obtain results for a particular system than for a ``typical'' or ``generic'' one. In the case of a specific substitution with
	$|\alpha_2|>1$ it was shown in \cite{BuSo14} that QWM holds with a log-H\"older rate for the self-similar suspension flow. Recently, analogous results were
	obtained for the substitution $\Z$-action, see \cite{BuMMSo25}. For Salem type substitutions, we are not aware of any QWM results in the literature.
	Considering the self-similar suspension flow over a Salem substitution, the first-named author obtained in \cite{Marshall20} H\"older spectral bounds of the form \eqref{spec-bound}, with $h(r) = r^{\gam(\om)}$, but only for algebraic spectral	parameters $\om\in \Q(\alpha)$ and with $r_0=r_0(\om)$. This was extended
	to the case of $\Z$-actions in \cite{BuMMSo25} as well.
	These bounds are, unfortunately, not strong enough to yield any QWM rate.
	It turns out that a very weak, $\log^*$ QWM rate for the self-similar suspension flow may be deduced essentially by the same method as in \cite{BuSo14}.
	Although we  do not expect it to be sharp by any means, arguably, any quantitative rate is better than no rate. We sketch the derivation of this
	result in the 
	Appendix.
	
	A connection between spectral bounds for substitution flows and the decay rate of the Fourier transform for Bernoulli convolution measures (a much-studied
	class of self-similar ``fractal'' measures on the line, see \cite{sixty,Varju}),  was pointed out
	 in \cite{BuSo14}. In the Appendix we recall this connection and note that the proof of QWM with a $\log^*$ rate also yields
	$\log^*$ Fourier decay for Salem Bernoulli convolutions.
	
	The rest of the paper is organized as follows. In the next section we present the background in more detail and state our main result. In Section 3 we prove the
	main theorem, and in Section 4 we show the absence of uniform H\"older bounds. We conclude in the Appendix as mentioned above.	

	
\section{Background and statement of results}
	\subsection{Suspension flows.} 
Let $\zeta$ be a primitive non-periodic substitution on $\A = \{1, \dots, d\}$, with a substitution matrix $\Sf_\zeta$.
 We obtain the substitution dynamical system $(X_\zeta,T)$, which is minimal and uniquely ergodic, with a invariant Borel probability measure
 $\mu$. Recall the definition of suspension flows
Let $\vec{s} = (s_1,\dots,s_d) \in \R_+^d$ and $F: X_\zeta \times \R \longrightarrow X_\zeta \times \R$ defined by $F(x,t) = (T(x),t-s_{x_0})$. The {\em suspension flow with the roof vector} $\vec{s}$, is defined as
	\[
	X^{\vec{s}}_{\zeta} = (X_\zeta \times \R) /\sim
	\] 
	where $\sim$ is the equivalence relation defined by $(x,t) \sim (x',t')$ if and only if $F^n(x,t) = (x',t')$, for some $n \in \Z$. There is a natural $\R$-action on $X^{\vec{s}}_{\zeta}$ given by the ``vertical" flow $(\phi_t)_{t\in\R}$: $\phi_t(x,t') = (x,t'+t)\, (\text{mod}\sim)$. From now on we identify $X^{\vec{s}}_{\zeta}$ with the {\em fundamental domain} of points $(x,t)\in X_\zeta \times \R$ such that $0\leq t < s_{x_0}$.
The $\R$-action $(	X^{\vec{s}}_{\zeta}, \phi_t)$ is uniquely ergodic as well, with the unique invariant Borel probability measure $\wt \mu_{\vec s}$
 defined,
up to normalization, as
$\mu|_{[a]}\times \Lk^1|_{[0,s_a]}$, when restricted to $[a]\times [0, s_a]$,
where $\Lk^1$ is the Lebesgue measure on the line, $a\in \A$, and $[a]$ denotes the cylinder set.

	We say that a function $f\in L^2(X_\zeta,\mu)$ is {\em cylindrical of level 0} if it depends only on $x_0$, the 0-th term of the sequence $x\in X_\zeta$. Cylindrical functions of level 0 form a $d$-dimensional vector space, with a basis
	$\{\One_{[a]}: \, a\in \Ak\}$. Denote by $\Cyl(X_\zeta)$ the space of cylindrical functions of level zero.
{\em Lip-cylindrical functions} for a suspension flow are defined by
\be \label{eq-cylinder}
f(x,t) = \sum_{j\in \A} \One_{[j]} \cdot \psi_j(t),\ \ \ \mbox{where}\ \ \psi_j \ \mbox{is Lipschitz on}\ [0,s_j].
\ee
Denote by $\text{Lip}(X^{\vec{s}}_\zeta)$ the space of Lip-cylindrical functions on $X^{\vec{s}}_\zeta$.


	\subsection{Spectral theory: general flows} Let $(X,\phi_t,\nu)$ be any probability-preserving $\R$-action. 
		Recall that for $f,g\in L^2(X,\nu)$ the (complex) spectral measure $\nu_{f,g}$ on $\R$ is determined by the Fourier transform:
	\begin{eqnarray*}
	\widehat{\nu}_{f,g}(-t) = \int_{\R} e^{2\pi i t \om}\,d\nu_{f,g}(\om)=
		\langle f\circ \phi_t\, , g\rangle,\quad t\in \R,
	\end{eqnarray*}
	where $\langle \cdot,\cdot \rangle$ denotes the scalar product in $L^2(X,\nu)$. We write $\nu_f = \nu_{f,f}$. 
		
	Closely related to spectral measures  are {\em twisted Birkhoff integrals}. For  $R\geq 1$, $f\in L^2(X,\nu)$,
	$\om\in \R$, and $y\in X$, it is  defined as
	\[
	S^f_R(y,\om) = \int_0^R f(\phi_t y) e^{2\pi i t \om} dt.
	\]
	
	Set $G_R(f,\om) = \dfrac{1}{R}\|{S^f_R(\cdot,\omega)}\|^2_{_{L^2(\nu)}}$.
	The next lemma is essentially from \cite{Hof}, see \cite[Lemma 4.3]{BuSo14} for a proof.
	
	\begin{lemma} \label{lem-easy1}
		For all $\om\in\R$ and $r\in (0,\half]$ we have
		\be \label{eq-estim1}
		\sig_f(B_r(\om))\le \frac{\pi^2}{4R} G_R(f,\om),\ \ \ \mbox{with}\ \ R = (2r)^{-1}.
		\ee
	\end{lemma}
	
	\subsection{Spectral theory: substitution flows}
	A useful general bound for twisted Birkhoff integrals in the context of substitution $\R$-actions was obtained in \cite[Prop.\,4.4]{BuSo14} and
	\cite[Prop.\,3.2(ii)]{BuSo18b}. In order to state it, we need to introduce another definition.

	A word $v\in \Ak^+$ is called a {\em return word} for $\zeta$ if $v$ starts with a letter $c$ for some $c\in \Ak$ and $vc$ is admissible for the language of $X_\zeta$. From the uniform recurrence of $(X_\zeta,T)$ it follows that the set of {\em irreducible} return words (i.e., those which cannot be
	split into a concatenation of shorter return word) is finite and non-empty. 
		A word	 $v$ is called a {\em good return word} for the substitution $\zeta$ if $v$ starts with some letter $c\in \Ak$ and
		$vc$ occurs as a subword in $\zeta(b)$ for every $b\in \Ak$. Denote the set of good return words by $GR(\zeta)$.
		Denote by 
	$\ell(v) = [\ell(v)_j]_{j\le d}$ the {\em population vector} of $v\in \Ak^+$, where $\ell(v)_j$ is the number of letters $j$ in $v$.
	
	\begin{prop}[{\cite[Prop.\,3.2(ii)]{BuSo18b}}]          \label{prop:prod}
		There exist $\lambda \in (0, 1)$, $C_1>0$ and depending only on the substitution $\zeta$, such that for all $f\in \Lip(\Xxi)$,
		\[
		\left|S^f_R((x,t),\omega)\right| \leq C_1\|f\|_\infty  \cdot \min\{1, |\om|^{-1}\} \cdot R\prod_{n=0}^{\floor{\log_{\alpha}R}} \left( 1-\lambda \max_{v\in GR(\zeta)}{\big\| \bigl\langle {(\Sf_\zeta^{\sf T})}^n \om \vec s, \ell(v)\bigr\rangle\bigr\|}^2_{_{\R/\Z}}\right) ,
		\]
		for all $R>1$, $(x,t) \in \Xxi$, and $\omega\in\R$.
	\end{prop}

	(In \cite[Prop.\,3.2(ii)]{BuSo18b} the product is up to $\floor{\log_\alpha R}-C_2$, for some $C_2>0$ depending on $\zeta$, but this can be 
	``absorbed'' into $C_1$ since $\lam<1$ and $\|z\|\le \half$ for all $z\in \R$.)

		Consider the Abelian group $\Gam$ (a subgroup of $\Z^d)$ generated by the population vectors of all return words for $X_\zeta$.
		We have $\Sf_\zeta\Gam\subset \Gam$,
		hence the real span of $\Gam$ is a rational $\Sf_\zeta$-invariant subspace. By assumption, $\Sf_\zeta$ is irreducible, hence $\Gam$ has full rank, i.e.,
		 it is a lattice. Passing to a sufficiently high power of the substitution, which does not change the substitution space,
		 we can assume that every ``elementary'' return word (i.e., not a concatenation
		 of shorter return words) is a good return
		 word, and we obtain from Proposition~\ref{prop:prod}, combined with \cite[Lemma 4.2]{BuMMSo25}:
		 \beq\label{eq:product}	\left|S^f_R((x,t),\omega)\right| \leq C_1\|f\|_\infty  \cdot \min\{1, |\om|^{-1}\}
		 \cdot R\prod_{n=0}^{\floor{\log_{\alpha}R}} \left( 1-\lambda \cdot
		\bigl\|(\Sf_\zeta^{\sf T})^n \om \vec s\bigr\|^2_{_{\R^d/\Gam^*}}\right),\ R>1.
		 \eeq

\subsection{Statement of results} Let $\zeta$ be a substitution of Salem type, i.e., the PF eigenvalue of the substitution matrix $\Sf_\zeta$
is a Salem number and $\Sf_\zeta$ is irreducible over $\Q$. Our main result is as follows.

\begin{theorem}\label{th:main}
There exists $\beta = \beta(\zeta)>0$ such that 
for Lebesgue-a.e.\ $\vec s\in \R^d_+$, for every Lip-cylindrical function $f$ of mean zero, the spectral measure
$\sig_f$ of the suspension flow $(\Xxi, \phi_t,\wt\mu_{\vec s})$ satisfies
\beq \label{eq:main}
\sig_f(B_r(\om)) \le \wt C_{\zeta,\vec s} \|f\|^2_\infty \cdot h_\beta(r),\ \ r>0,\ \ \ \mbox{where}\ \ 
h_\beta(r)  = \exp\left(-\frac{\beta \log\log\log_\alpha (1/2r)}{\log\log\log \log_\alpha (1/2r)}\right),
\ee
for all $\om\in \R$. Thus the system has QWM with the rate $h_\beta(r)$ for such functions $f$.
\end{theorem}

We do not know if this is sharp; QWM with a log-H\"older rate seems more plausible. In any case, we know that H\"older QWM rate is impossible
for any $\vec s$, see the next theorem.
The idea of the proof of Theorem~\ref{th:main} is based on a variant of the Erd\H{o}s-Kahane technique, in the framework of
``vector expansions in base $\Sf_\zeta$'', developed in 
\cite{BuMMSo25}. As in many other recent papers on QWM, the proof proceeds  via \eqref{eq:product}, followed by Lemma~\ref{lem-easy1}. Thus we need to consider
the distance from $(\Sf_\zeta^{\sf T})^n \om \vec s$ to the nearest point in the lattice $\Gam^*$. We ``win'', so to say, if this distance is not too small
some of the time, in a quantitative way, which has  to be uniform in the spectral parameter $\om$. 
However, we do have some freedom in the choice of the roof vector $\vec s$, since we are 
allowed to throw away a set of measure zero.

\begin{theorem} \label{th:noQWM}
For any substitution $\zeta$ of Salem type on $d$ symbols, for any $\vec s\in \R^d$ and any $\alpha>0$, for almost every Lip-cylindrical function $f$
of mean zero, for any $C_f$, the inequality \eqref{spec-bound} is violated for $h(r) = r^\alpha$, for a sequence of $R = R_n\to \infty$.
\end{theorem}

The proof of Theorem~\ref{th:noQWM} is analogous to the proof of  \cite[Theorem 2.7]{BuMMSo25}, which deals with the cases when
$\vec s = \one$ and when $\vec s$ is the PF eigenvector.


\section{Proof of the Theorem~\ref{th:main}}

\subsection{The set-up}
By assumption, we have a substitution $\zeta$ on $d$ symbols with an irreducible over $\Q$ substitution matrix ${\sf S_\zeta}$ having a Salem number
 $\alpha$ as the PF eigenvalue. The other eigenvalues are the Galois conjugates of $\alpha=\alpha_1$: they are
 $\alpha_2=\alpha^{-1}$, and the rest are $\alpha_3,\ldots,\alpha_d$, complex conjugate pairs 
of modulus one. 
Fix a height vector $\vec s$ and thus the corresponding suspension flow $(\Xxi, \phi_t,\wt\mu_{\vec s})$.
Let $A ={\sf S_\zeta^T}$ and we consider its action on $\C^d$.
The matrix $A$ is diagonalizable over $\C$; we are going to use the decomposition of (real) vectors in terms of the 
(complex) basis of eigenvectors of $A$. Let $\{\bbe^*_j\}_{j\le d}$  be a basis consisting of eigenvectors for the transpose matrix $A^{\sf T}$, and let
$\{\bbe_j\}_{j\le d}$ be the dual basis, i.e., $\langle \bbe_k,\bbe_j^*\rangle = \delta_{kj}$. If $\bx = \sum_{k=1}^d c_k \bbe_k\in \C^d$, then 
$c_k = \langle x,\bbe_k^*\rangle$ and 
\be\label{eigen1}
\langle A^n \bx, \bbe_k^*\rangle = \alpha_k^n c_k,\ \ k=1,\ldots, d.
\ee

Fix $B>1$ and then $\om\ne 0$, so that $|\om|\in [B^{-1},B]$.
Consider the lattice $L:=\Gam^*<\R^d<\C^d$, where $\Gam$ is generated by the population vectors of return words,
see Section 2.3.
Note that  $L$ depends only on $\zeta$, and 
$AL\subset L$. Below we will use $\ell^\infty$ metric and matrix norm, unless stated otherwise.

We will run the scheme of \cite[Section 3]{BuMMSo25}, replacing the vector $\vec 1$ by $\vec s\in \R^d_+$, namely, 
 we write
$$A^n \om \vec s = \bp_n + \beps_n,\ \ n\ge 0,$$
where $\bp_n$ is the nearest lattice point, so that
$$
\|\beps_n\|=\bigl\|A^n \om \vec s\bigr\|_{_{\R^d/\Gam^*}}=\bigl\|(\Sf_\zeta^{\sf T})^n \om \vec s\bigr\|_{_{\R^d/\Gam^*}}.
$$
Note that $\bp_n$  and $\beps_n$ depend on $\om$.

In order to keep the paper self-contained, we repeat the first steps of \cite[Section 3]{BuMMSo25}.
Define
	$$
	\bz_0 := -\bp_0,\ \ \bz_n:= A\bp_{n-1} - \bp_{n},\ \ n\ge 1.
	$$
	We have $A(\bp_{n-1} + \beps_{n-1}) = \bp_n + \beps_n$, hence
	\be \label{eq11}
	\bz_n= A\bp_{n-1} - \bp_{n} = \beps_{n} - A \beps_{n-1} \in L,\ \ n\ge 1.
	\ee
The following is now immediate:

	\begin{lemma} \label{lem1}
	Let $a_L>0$ be the minimal distance between distinct points in $L$ and let $b_L>0$ be such that the union of balls of radius $b_L$ with centers in $L$ covers $\R^d$.  Then
\begin{enumerate}
\item[(i)] 	
${\|\beps_n\|} \le b_L\ \ \mbox{for all}\ \ n\ge 0$;
\item[(ii)] 
	$
	\| \bz_n\|  \le (1 + \|A\|)\cdot b_L \ \ \mbox{for all}\ n\ge 0,
	$
hence there exists a finite set $F\subset L$ such that $\bz_n\in F$ for $n\ge 0$;
\item[(iii)] if
 $\max\{\|\beps_{n-1}\|, \|\beps_{n}\|\} < \frac{a_L}{4\|A\|}=: c_{A,L}$, then $\bz_n = \bf{0}$ and $\beps_{n} = A\beps_{n-1}$.
\end{enumerate}
\end{lemma}
	

\medskip

In view of \eqref{eq11}, we can write for $n\ge 1$:
	\begin{eqnarray*}
		\beps_{n} = \bz_n + A \beps_{n-1} & = & \bz_n + A(\bz_{n-1} + A\beps_{n-2})\\
		& = & \bz_n + A\bz_{n-1} + A^2\beps_{n-2} = \ldots \\
		& = & \bz_n + A\bz_{n-1} + \cdots + A^{n-1}\bz_{1} + A^{n} \beps_0,
	\end{eqnarray*}
hence
\be \label{eq101}
\om \vec s = -\bz_0 - A^{-1}\bz_1 - A^{-2}\bz_2 - \cdots -A^{-n}\bz_n + A^{-n} \beps_n,\ \ n\ge 1,
\ee
since $$\beps_0 = \om\vec s - \bp_0 = \om\vec s + \bz_0.$$
We view \eqref{eq101} as a ``vector digit expansion in base $A$''; note that 
$\bz_j\in F$, a finite subset of $L$.

Denote $b_n^{(j)}:= \langle \bz_n, \bbe^*_j\rangle$ for $n\ge 0$, and write
$$
\Phi_n^{(j)}(x) = - b_0^{(j)} - b_1^{(j)}x - \cdots - b_n^{(j)} x^n.
$$
Now \eqref{eq101} implies
\be \label{eq2}
\om\langle \vec s, \bbe_j^*\rangle = \Phi^{(j)}_n(1/\alpha_j) + 
\frac{\langle \beps_n,\bbe_j^*\rangle}{\alpha_j^n},\ \ j=1,\ldots,d,
\ee
where
\be \label{eps-bound2}
|\langle \beps_n,\bbe_j^*\rangle|\le C_0\|\beps_n\| \le C_1 = C_1(A,L)\ \ \mbox{for all}\ \ n,j.
\ee
Thus $C_1$ depends only on $\zeta$.
We can assume that $C_1\ge 1$. Without loss of generality, $\langle \vec s, \bbe_1^*\rangle=1$, since $\bbe_1^*$ is a positive PF eigenvector, hence
\be \label{om-eq}
\om = \Phi^{(1)}_n(1/\alpha) + \frac{\langle \beps_n,\bbe_1^*\rangle}{\alpha^n}.
\ee
We obtain
\be \label{eq2a}
\langle \vec s, \bbe_3^*\rangle = 
\frac{\Phi^{(3)}_n(1/\alpha_3) + \frac{\langle \beps_n,\bbe_3^*\rangle}{\alpha_3^n}}{\om} = 
\frac{\Phi^{(3)}_n(1/\alpha_3) + \frac{\langle \beps_n,\bbe_3^*\rangle}{\alpha_3^n}}{\Phi^{(1)}_n(1/\alpha) + \frac{\langle \beps_n,\bbe_1^*\rangle}{\alpha^n}}\,
\ee
(in fact, any $\bbe_j^*$ for $j=3,\ldots,d$ could be used).
Observe that 
$$
|\Phi^{(1)}_n(1/\alpha)|\ge B^{-1} - C_1\alpha^{-n}
$$
by \eqref{om-eq} and \eqref{eps-bound2}, since $|\om|\ge B^{-1}$.
Now\eqref{eq2a} yields, in view of $|\alpha_3|=1$:
\begin{eqnarray}
\left|\langle \vec s, \bbe_3^*\rangle - \frac{\Phi^{(3)}_n(1/\alpha_3)}{\Phi^{(1)}_n(1/\alpha)}\right| & \le &
\frac{|\langle \beps_n,\bbe_3^*\rangle|}{\om}
+ \frac{|\Phi^{(3)}_n(1/\alpha_3)|\cdot |\langle \beps_n,\bbe_1^*\rangle|}
{\om \alpha^n |\Phi^{(1)}_n(1/\alpha)|}\nonumber \\[1.1ex]
& \le & B |\langle \beps_n,\bbe_3^*\rangle| + \frac{C_1 C_2 B(n+1)}{\alpha^n (B^{-1} - C_1\alpha^{-n})}, \label{eq3}
\end{eqnarray}
where $C_2 = \max_{k,j}\|b_k^{(j)}\| = C_2(A,L)=C_2(\zeta)$.

Note that $\bbe_3^*$ is a complex vector; the inequality \eqref{eq3} is equivalent to two real inequalities.
Identical inequalities are obtained if we use the eigenvalue $\ov\alpha_3$, corresponding to the complex conjugate eigenvector $\ov\bbe_3^*$.
Let $|\bx|_3 := |\langle \bx, \bbe_3^*\rangle|$; note that $|A\bx|_3 = |\bx|_3$ for a real vector $\bx\in \R^d$ by \eqref{eigen1}.
Recall that we have
\be \label{eq-step}
\bz_n={\bf 0}\ \implies \ \beps_n = A\beps_{n-1}\ \implies {|\beps_n|}_3 = {|\beps_{n-1}|}_3.
\ee
Further, let $\Proj_3$ be the projection, commuting with $A$, onto
the {\em real} $A$-invariant plane, on which $A$ acts a rotation, with eigenvalues $\alpha_3$ and $\ov\alpha_3$ (in appropriate coordinates).

\begin{remark} \label{rem}
It is easy to see that if we know $\langle \beps_n,\bbe_3^*\rangle\in \C$ up to some error $\delta$, for a real vector $\vec s$, then we know $\proj_3(\vec s)$
up to an error $\const\cdot \delta$, with a constant depending only on $A$. 
\end{remark}

\subsection{The main proposition}
The estimate \eqref{eq:main} is obtained first for $\om\in [B^{-1},B]$, for any Lip-cylindrical $f$, with an explicit dependence on $B>1$, which is then
``glued'' with a standard estimate for $\om=0$ and $f$ of mean zero.


\begin{prop}\label{prop:main}
There exists $\beta = \beta(\zeta)>0$ such that for Lebesgue-a.e.\ $\vec s\in \R^d_+$, 
for every Lip-cylindrical function $f$, the spectral measure
$\sig_f$ of the suspension flow $(\Xxi, \phi_t,\wt\mu_{\vec s})$ satisfies
\be \label{bound44}
\sig_f(B_r(\om)) \le C_{\zeta,\vec s} \|f\|^2_\infty \cdot \min\{1, |\om|^{-2}\}\cdot 
\exp\left(-\frac{\beta \log\log\log_\alpha (1/2r)}{\log\log\log \log_\alpha (1/2r)}\right),\ \ r\le (2R_0)^{-1},
\ee
for all $\om\ne 0$, where
\be \label{def-R0}
R_0 = \alpha^{2^{\scriptstyle{\lfloor\gam \log(B)/(\log\log B)\rfloor}!}},\ \ \mbox{with}\ \  \gam = \gam (\zeta,\vec s)\ \ \mbox{and}
\ \ B = \max\{16, \lceil|\om|\rceil, \lceil|\om|^{-1}\rceil\}.
\ee
\end{prop}
Denote
\be \label{def:func}
\Psi(\tau):= \frac{\log\tau}{\log\log\tau},
\ee
where $\tau \ge 4$,
so that the formula for $R_0$ in \eqref{def-R0} can be written as
\be \label{def1-R0}
R_0 = \alpha^{2^{\scriptstyle{\lfloor \gam\cdot \Psi(B) \rfloor !}}}.
\ee

\begin{proof}
By ``scale $n$'' we mean the range $R\sim \alpha^n$.
It is convenient to consider the sequence of scales
$$
n_k:= 2^{k!},
$$
so that $n_1 = 2, n_2 = 2^{2}, n_3 = 2^{6},\ldots$

Now we define the exceptional set. Recall that the sequences $\bz_n$ and $\beps_n$ depend on $\om$ and $\vec s$. First we define the ``bad set at scale $n_k$'' as follows:
\begin{eqnarray*}
\Ek_k = \Ek_k(B) & = &    \left\{\vec s\in \R^d_+,\ \langle \vec s, \bbe_1^*\rangle=1:  \ \exists\,\om\ne 0,\ |\om|\in [B^{-1},B],\ \ \mbox{satisfying:}\right.\\
& & \ \ \ \ \ \ 
 \mbox{(E1)} \ \ \#\{\ell \le n_k,\ \bz_\ell\ne 0\} < k/2;\\
 & & \ \ \ \ \ \
 \mbox{(E2)} \ \ \bz_\ell=0\ \mbox{for all}\ \ell\in (n_{k-1},n_k);\\
 & & \left. \ \ \ \ \ \
 \mbox{(E3)} \ \ |\beps_{n_{k-1}}|_3 < (2k/n_k)^{1/2}\right\}.
\end{eqnarray*}
And then let
\be \label{def-excep1}
\Ek = \Ek(\Ups) = \bigcap_{k_0=1}^\infty \bigcup_{B=16}^\infty\ \bigcup_{k=k_0+\lfloor \Ups \cdot\Psi(B)\rfloor}^\infty \Ek_k(B).
\ee
The parameter $\Ups\ge 1$  will be fixed later.
Because of our normalization, admissible roof vectors $\vec s$ lie in
a $(d-1)$-dimensional simplex in $\R^d_+$. Denote by $\Lk^r$ the $r$-dimensional Lebesgue measure.

\begin{lemma}
There exists $\Ups\ge 1$ such that $\Lk^{d-1}(\Ek(\Ups)) = 0$.
\end{lemma}

\begin{proof}[Proof] We write $\Ek = \Ek(\Ups)$, but keep the dependence on $\Ups$ in mind.
In  is enough to show that $\Lk^2(\proj_3(\Ek))=0$. In view of Remark~\ref{rem}, this is equivalent to $\Lk^2(\Pr_3(\Ek))=0$, where $\Pr_3(\bx) =
\langle \bx,\bbe_3^*\rangle \in \C$. Clearly, 
\be \label{excep2}
\Pr_3(\Ek) \subseteq \bigcap_{k_0=1}^\infty \bigcup_{B=16}^\infty\ \bigcup_{k=k_0+\lfloor \Ups\cdot \Psi(B)\rfloor}^\infty \Pr_3(\Ek_k(B)).
\ee
Thus it suffices to  get an efficient cover of the set of projections $\Pr_3(\Ek_k(B))$ for $k$ sufficiently large.
Let $\vec s \in \Ek_k(B)$. This means that there exists $\om$ with $|\om|\in [B^{-1},B]$, such that the properties (E1)-(E3) hold.
We fix such an $\om$ and use \eqref{eq3} for $n=n_{k-1}$:
\begin{eqnarray}
\left|\langle \vec s, \bbe_3^*\rangle - \frac{\Phi^{(3)}_{n_{k-1}}(1/\alpha_3)}{\Phi^{(1)}_{n_{k-1}}(1/\alpha)}\right| & \le & 
B |\langle \beps_{n_{k-1}},\bbe_3^*\rangle| + \frac{C_1 C_2 B(n_{k-1}+1)}{\alpha^{n_{k-1}} (B^{-1} - C_1\alpha^{-n_{k-1}})} \nonumber \\[1.1ex]
& \le & B^2\left(\frac{2k}{n_k}\right)^{1/2} + \frac{C_1 C_2 B(n_{k-1}+1)}{\alpha^{n_{k-1}} (B^{-1} - C_1\alpha^{-n_{k-1}})}, \label{eq34}
\end{eqnarray}
where the second inequality follows from $B\ge 2$ and from (E3).

We want to make sure that the second term in \eqref{eq34} is dominated by the first one.
Note that if $\alpha^{n_{k-1}}\ge 2BC_1$, then 
\be \label{cond1}
\frac{C_1 C_2 B(n_{k-1}+1)}{\alpha^{n_{k-1}} (B^{-1} - C_1\alpha^{-n_{k-1}})} \le \frac{C_1 C_2 B^2(n_{k-1}+1)}{ \alpha^{n_{k-1}}}.
\ee
We have $n_{k-1} = 2^{(k-1)!}\ge k^2$ for $k\ge 4$, hence $\alpha^{k^2}  \ge 2BC_1$ will guarantee \eqref{cond1} for $k\ge 4$.
In the definition \eqref{def-excep1} we have
$k\ge \Ups \cdot \frac{\log B}{\log\log B}\ge \Ups\cdot (\log B)^{1/2}$ for $B\ge 16$, hence
$$
\alpha^{k^2} \ge \alpha^{\Ups^2\log B} = B^{\Ups^2 \log\alpha} \ge 2BC_1,
$$
for $\Ups \ge \Ups_0(\zeta)$.
It remains to check that
$$
\frac{C_1 C_2 B^2(n_{k-1}+1)}{ \alpha^{n_{k-1}}} \le B^2\left(\frac{2k}{n_k}\right)^{1/2},
$$
for $k\ge k_0(\zeta)$ (most importantly, independent of $B$), which is easy to see from the fact that
$n_k = 2^{k!} \ll \alpha^{2^{(k-1)!}} = \alpha^{n_{k-1}}$ for large $k$.
Finally, all of the above 
 imply that
$$
\left|\langle \vec s, \bbe_3^*\rangle - \frac{\Phi^{(3)}_{n_{k-1}}(1/\alpha_3)}{\Phi^{(1)}_{n_{k-1}}(1/\alpha)}\right|\le 2B^2\left(\frac{2k}{n_k}\right)^{1/2}
$$
for $k\ge k_0+\lfloor \Ups\cdot \Psi(B) \rfloor$, with $k_0 = k_0(\zeta)$ and $\Ups\ge \Ups_0(\zeta)$.
In view of Remark~\ref{rem}, $\Pr_3(\vec s)$ belongs to the disk of radius $2C_3B^2\cdot(2k/n_k)^{1/2}$, centered at the point given by the real and
imaginary parts of the complex number $\displaystyle{\frac{\Phi^{(3)}_{n_{k-1}}(1/\alpha_3)}{\Phi^{(1)}_{n_{k-1}}(1/\alpha)}}$, in
the invariant plane corresponding to $\alpha_3$, for some $C_3 = C_3(\zeta)$.
In view of (E1), the number of possible such expressions is bounded by the number of subsets of $[1,n_{k-1}]$ of cardinality $< k/2$, times
$(\# F)^{\frac{k}{2}}$.Thus,
$$
\Lk^2(\proj_3(\Ek_k)) \le 4C^2_3 B^4 \sum_{1\le i < k/2} {n_{k-1}\choose i}\cdot (\# F)^{k/2}\cdot \frac{2k}{n_k}\,.
$$
Using a crude upper bound 
$$
\sum_{1\le i <\ell/2} {m\choose i} < {m\choose \ell} < \frac{m^\ell}{\ell!} < \Bigl(\frac{m}{\ell}\Bigr)^\ell e^\ell\ \ \mbox{ for}\ \  m\gg \ell>1,
$$ 
we obtain for $k\ge k_0+\lfloor \Ups\cdot \Psi(B) \rfloor$:
\begin{eqnarray*}
\Lk^2(\proj_3(\Ek_k)) & \le  & 4C^2_3 B^4 \left(\frac{2^{(k-1)!}}{k}\right)^k\cdot (\# F \cdot e)^{k}\cdot \frac{2k}{2^{k!}} 
= 8C^2_3 B^4 \frac{(\# F \cdot e)^k}{k^{k-1}}\,.
\end{eqnarray*}
Next, assuming that $k_0=k_0(\zeta)$ is sufficiently large, we have
$$
\frac{(\# F \cdot e)^k}{k^{k-1}} < 2^{-k} k^{-k/2} < 2^{-k_0} k^{-k/2}, \ \ \ \mbox{for}\ \ \ k\ge k_0,
$$
and hence
\begin{eqnarray}
\Lk^2\left(\bigcup_{k=k_0+\lfloor \Ups\cdot \Psi(B)\rfloor}^\infty \Pr_3(\Ek_k(B))\right) & \le & 8C^2_3 B^4 \cdot 2^{-k_0} \cdot\!\!\!\!\!\! \nonumber
\sum_{k=\lfloor \Ups\cdot \Psi(B)\rfloor}^\infty k^{-k/2} \\[1.1ex]
& < & C_4B^4\cdot 2^{-k_0} (\Ups\cdot \Psi(B))^{-\Ups\cdot \Psi(B)/2}. \label{bound47}
\end{eqnarray}
for some $C_4 = C_4(\zeta)$, 
since the series decays super-exponentially.

\begin{claim} \label{claim}
We have
\be \label{eq:claim}
\Psi(B)^{\Psi(B)} \ge B^{1/2},\ \ \ \mbox{for all}\ \ B\ge 16.
\ee
\end{claim}

\begin{proof}[Proof of the claim]
Note that $\Psi(B)$ is positive for $B\ge 16$.
We have
$$
\log\Bigl(\Psi(B)^{\Psi(B)}\Bigr) = \frac{\log B}{\log\log B} \cdot \log\Bigl(\frac{\log B}{\log\log B}\Bigr) = \log B 
\Bigl(1 - \frac{\log\log\log B}{\log\log B}\Bigr) \ge \half \log B,\ \ B\ge 16,
$$
and the claim follows.
\end{proof}

We will have $\Ups\ge 2$, so that 
$$
(\Ups\cdot \Psi(B))^{-\Ups\cdot \Psi(B)/2} \le \Psi(B)^{-\Ups\cdot \Psi(B)/2} \le B^{-\Ups/4},
$$
in view of the claim.
Together with \eqref{bound47}, this yields
$$
\Lk^2\left(\bigcup_{k=k_0+\lfloor \Ups\cdot \Psi(B)\rfloor}^\infty \Pr_3(\Ek_k(B))\right)\le C_4B^4 \cdot 2^{-k_0} B^{-\Ups/4}.
$$
Assuming that $\Ups \ge 24$, we obtain
$$
\Lk^2\left(\bigcup_{B=16}^\infty\ \bigcup_{k=k_0+\lfloor \Ups\cdot \Psi(B)\rfloor}^\infty \Pr_3(\Ek_k(B))\right) < C_4 \cdot 2^{-k_0}
\sum_{B=16}^\infty B^{4-\Ups/4} < C_4 (\pi^2/6)\cdot 2^{-k_0},
$$
using the formula for the sum of the series $\sum n^{-2}$. The latter tends to $0$ as $k_0\to \infty$, and hence $\Lk^2(\Pr_3(\Ek))=0$ by \eqref{excep2}, 
as desired. The proof of the lemma
is complete.
\end{proof}

Next we are going to show local estimates for spectral measures of Lip-cylindrical functions, for suspension flows with $\vec s\notin \Ek$.
Suppose that $\vec s\notin \Ek$, then there exists $k_0=k_0(\vec s)\in \N$ such that $\vec s\notin \Ek_k(B)$ for all $B\ge 16$ and 
$k\ge k_0 + \lfloor\Ups \cdot \Psi(B)\rfloor$. Fix an $\om$, with $|\om|\in [B^{-1},B]$. By the definition of $\Ek(B)$, one of the properties
(E1), (ii), (E3)  is violated.

Consider a scale $n_k$ for $k \ge 2(k_0 + \lfloor\Ups \cdot\Psi(B)\rfloor$), more precisely,  assume that
$$
\alpha^{2^{k!}} \le R < \alpha^{2^{(k+1)!}}.
$$
We copy \eqref{eq:product} for the reader's convenience:
$$
\left|S^f_R((x,t),\omega)\right| \leq C_1\|f\|_\infty  \cdot \min\{1, |\om|^{-1}\}
		 \cdot R\prod_{n=0}^{\floor{\log_{\alpha}R}} \left( 1-\lambda \cdot
		\bigl\|(\Sf_\zeta^{\sf T})^n \om \vec s\bigr\|^2_{_{\R^d/\Gam^*}}\right).
		$$

If (E1) is violated, then $\#\{\ell \le n_k: \bz_\ell\ne 0\} \ge k/2$, hence, in view of Lemma~\ref{lem1}(iii), we have
$\#\{\ell\le n_k: \|\beps_\ell\| \ge c_{A,L}\} \ge k/4$, which implies
\be \label{bound1}
\left|S^f_R((x,t),\omega)\right| \leq C_1\|f\|_\infty  \cdot \min\{1, |\om|^{-1}\}
		  \cdot R  \cdot (1-\lam\cdot c_{A,L}^2)^{k/4}.
\ee
We have 
\be \label{bound2}
\log\log_\alpha R <  (k+1)! < k^k \implies\log\log \log_\alpha R < k\log k \implies k> 
\frac{\log\log \log_\alpha R}{\log\log\log \log_\alpha R}
\ee
for $R\ge \alpha^{16}$, using an elementary fact:
\be \label{elem}
k\log k>T\implies k > \frac{T}{\log T}\ \ \mbox{for}\ \ T\ge 2.
\ee
Combining \eqref{bound2} with \eqref{bound1}, we obtain for $k\ge 2(k_0 + \lfloor\Ups \cdot \Psi(B)\rfloor)$:
\be \label{bound3}
 \left|S^f_R((x,t),\omega)\right| \leq C_1\|f\|_\infty  \cdot \min\{1, |\om|^{-1}\}
		 \cdot R  \cdot  \exp[-\beta_1 \Psi(\log\log_\alpha R)],\ \ 
R\ge R_0,
\ee
for $\beta_1= -\frac{1}{4} \ln(1-\lam\cdot c_{A,L}^2)$ and
$$
R_0 = \alpha^{2^{{\scriptstyle{2(k_0+\lfloor \Ups\cdot \Psi(B)\rfloor)!}}}}. 
$$
We can  simplify this by getting a slightly worse bound, using that 
$2(k_0+\lfloor \Ups\cdot \Psi(B)\rfloor)\le \lfloor\gam\cdot\Psi(B)\rfloor$, where $\gam = 4\max\{k_0+1,\Ups\}$, since $\Psi(B)\ge 2$ for $B\ge 16$.
Thus we have \eqref{bound3} with
\be \label{R0}
R_0 = \alpha^{2^{\scriptstyle{\lfloor\gam \cdot\Psi(B)\rfloor!}}}. 
\ee

Now we can apply Lemma~\ref{lem-easy1} to obtain \eqref{bound44}, as desired. Note that $\gam$ depends on $\vec s$, since 
$k_0$ depends on $\vec s$.

If (E1) holds at the scale $n_k$, then 
$$
\exists\, q\in (k/2,k):\  \bz_\ell=0\ \mbox{for all}\ \ell\in (n_{q-1},n_q).
$$
Recall that $\vec s\notin \Ek_q(B)$ by assumption, since $q\ge k/2 \ge k_0 + \lfloor\Ups\cdot \Psi(B) \rfloor$. If (E1) is violated at scale
$n_q$, we get a bound like \eqref{bound1}, only with $(1-\lam\cdot c_{A,L}^2)^{q/4} \le (1-\lam\cdot c_{A,L}^2)^{k/8}$, 
and we can conclude to obtain the analog of \eqref{bound44} 
similarly. The remaining case is when both (E1) and (E2) hold at scale $n_q$. Then (E3) must be false at that scale,
that is,
$$
|\beps_{n_{q-1}}|_3 \ge (2q/n_q)^{1/2}.
$$
Since (E2) holds, \eqref{eq-step} implies
that $|\beps_\ell|_3 =|\beps_{n_{q-1}}|_3 \ge (2q/n_q)^{1/2}$ for all $\ell\in (n_{q-1},n_q)$. Since $\|\beps_n\| \ge |\beps_n|_3$, we obtain from 
\eqref{eq:product} that
\begin{eqnarray}
 \left|S^f_R((x,t),\omega)\right| & \le & C_1\|f\|_\infty  \cdot \min\{1, |\om|^{-1}\}\cdot R \cdot \left(1 - \lam(2q/n_q)\right)^{n_q-n_{q-1}} \nonumber \\
                               & \le & C_1\|f\|_\infty  \cdot \min\{1, |\om|^{-1}\}\cdot R \cdot   \exp\bigl(-2\lam q(1 - \textstyle{\frac{n_{q-1}}{n_q}})\bigr) \nonumber   \\
                               & \le & C_1\|f\|_\infty  \cdot \min\{1, |\om|^{-1}\} \cdot R\cdot  \exp(-\lam k/2), \label{order}
\end{eqnarray}
because $n_{q-1}/n_q < 1/2$ for all $q\ge 2$, and \eqref{order}
 is again of the same order as \eqref{bound1}. In all cases the bound for $R_0$  is the same, since it comes from
$\alpha^{2^{k!}} \le R$ for  $k \ge 2(k_0 + \lfloor\Ups \cdot \Psi(B)\rfloor)$. Again we apply Lemma~\ref{lem-easy1}, and the proof is complete.
\end{proof}

\subsection{Conclusion of the proof: ``glueing'' near zero and at infinity} \label{sec:glue}

 Here we finally prove \eqref{eq:main}.
In view of $\sig_f(\R) = \|f\|_2^2\le \|f\|_\infty^2$, it is enough to obtain this inequality for $r\in (0,\rho_0)$, with a uniform constant 
$\rho_0$, since we can always ensure \eqref{eq:main} for $r\ge \rho_0$ by increasing the constant $\wt C_{\zeta,\vec s}$. Thus, we can immediately
assume that $|\om|\notin [\frac{1}{16}, 16]$, since otherwise \eqref{bound44} gives a required bound with a uniform $r_0$. Under this assumption we have
$$
B = B(\om) = \max\{\lceil|\om|\rceil, \lceil|\om|^{-1}\rceil\},
$$
see \eqref{def-R0}.
By \cite[Theorem 1]{Adam} (we are in case (iii) of Adamczewski's theorem), together with \cite[Prop.\,5.1]{BuSo18b}, whose proof is valid for $|\theta_2|=1$ as well, we obtain that if $f$ is a Lip-cylindrical function of mean zero, then 
\be \label{adam}
\sig_f(B_r(0)) \le C_5\|f\|_\infty^2\cdot (\log(1/r))^{2d-2}r^2\le C_5 \|f\|_\infty^2 \cdot r,\ \ r\in (0,\rho_0),
\ee
for some $C_5 = C_5(\zeta)$ and $\rho_0>0$, depending on $d$. This provides a much stronger bound than \eqref{eq:main} for $\om=0$.

Next we are going to ``glue'' this bound with \eqref{bound44} for $|\om|$ sufficiently small.
Thus, assume  that $|\om|\in (0,\frac{1}{16})$. We have by \eqref{adam}:
$$\sig_f(B_r(\om))\le \sig_f(B_{r + |\om|}(0)) \le C_5 \|f\|_\infty^2 (r+|\om|).$$
If $r>|\om|$, then $r+\om <  2r$, and  \eqref{eq:main} follows. If
\be \label{r-bound}
r \le r_0(\om) = (1/2) \alpha^{-2^{^{\scriptstyle{\lfloor\gam \cdot\Psi(B(\om))\rfloor!}}}},
\ee
then \eqref{bound44} implies the desired inequality. It remains to consider the case when $r_0(\om) < r \le |\om|$. Then \eqref{adam} implies
$$
\sig_f(B_r(\om)) \le \sig_f(B_{r + |\om|}(0)) \le 2C_5\|f\|_\infty^2\cdot |\om| \le 4C_5 \|f\|_\infty^2\cdot B(\om)^{-1},
$$
using that $B(\om) \le 2|\om|^{-1}$ for $|\om|\in (1,\frac{1}{16})$. In the next lemma we
are going to verify that this yields an upper bound of the same order as \eqref{bound44} for $|\om|$ and $r$ sufficiently small.

\begin{lemma} \label{lem:glue}
There exist $c_6 = c_6(\zeta,\vec s)\in (0,1)$ and a uniform constant $\rho_0>0$ such that if $B(\om)^{-1}\le c_6$ and $r\in (0,\rho_0)$, 
then
\be \label{om-bound}
r > r_0(\om) \implies B(\om)^{-1} <  \exp\left(-\frac{\beta \log\log\log_\alpha (1/2r)}{\log\log\log \log_\alpha (1/2r)}\right) =
\exp\bigl(-\beta \Psi(\log\log_\alpha(1/2r))\bigr).
\ee
\end{lemma}

\begin{proof} We write $B=B(\om)$ and prove the contrapositive. We have, denoting $\wt\beta = \beta/\ln 2$:
\be \label{imp1}
\exp\bigl(-\beta \Psi(\log\log_\alpha(1/2r))\bigr) \le B(\om)^{-1} \implies \wt\beta^{-1}\log B  \le \Psi(\log\log_\alpha(1/2r)).
\ee
Recall \eqref{elem}, which implies that
\be \label{elem2}
\Psi(T) \ge A \implies \log T \ge A\log A,\ \ \ \mbox{for}\ \ T\ge 4.
\ee
Denote $\Phi(A) = A\log A$.
Now, \eqref{imp1} and \eqref{elem2} yield
$$
\log\log\log_\alpha(1/2r) \ge \Phi(\wt\beta^{-1}\log B),\ \ \ \mbox{for}\ \ 0< r\le (1/2)\alpha^{-16},
$$
hence
$$
r \le (1/2) \alpha^{-2^{^{\scriptstyle{2^{^{\scriptstyle{\Phi(\wt\beta^{-1}\log B)}}}}}}}.
$$
In order to verify \eqref{r-bound}, it suffices to check that 
\be \label{bound51}
2^{\Phi(\wt\beta^{-1}\log B)} \ge \gam\Psi(B)^{\gam\Psi(B)},
\ee
equivalently,
$
\Phi(\wt\beta^{-1}\log B) \ge \Phi\bigl(\gam\Psi(B)\bigr).
$
Since $\Phi$ is an increasing function, the latter reduces to showing that
$$
\wt\beta^{-1}\log B\ge \gam\Psi(B) = \frac{\gam \log B}{\log\log B},
$$
which clearly holds for $B$ sufficiently large, depending on $\beta$ and $\gam$, that is, for $B(\om)^{-1}\le c_6$, for some constant $c_6 = c_6(\zeta,\vec s)\in
(0,1)$. This completes the proof of the lemma.
\end{proof}

It remains to consider the case when $|\om|>16$, which is actually analogous, since we can
rely on the factor $\min\{1, |\om|^{-1}\}$ in \eqref{eq:product}. If $B(\om)^{-1} > c_6$, then\eqref{def-R0} and \eqref{bound44} provide the desired bound with a
uniform $r_0(\om)$. If $r\le r_0(\om)$, again \eqref{bound44} applies. The remaining case is when $B(\om)^{-1} \le  c_6$, $r>r_0(\om)$, 
and we can assume without loss of generality that $r \in (0,\rho_0)$. But then Lemma~\ref{lem:glue} implies the desired inequality.

The proof of \eqref{eq:main} is complete. Finally, the assertion on quantitative mixing follows from Remark~\ref{rem:QWM}. This finishes the proof of 
Theorem~\ref{th:main}.




\section{Proof of lower bounds (no H\"older QWM for any suspension flow)}

\subsection{Some algebraic number theory (see e.g.\,\cite{Bug04})}.	A \textit{number field} is a finite field extension $K$ of $\Q$, and the \textit{degree} of $K$ is the dimension of $K$ as a vector space over $\Q$. An element $\alpha\in K$ is an \textit{algebraic integer} if it is a root of a monic polynomial with coefficients in $\Z$. Algebraic integers in $K$ form a ring denoted by $\mathcal{O}_K$.
	
	Let $K_1,K_2$ be two number fields. A $\Q$-\textit{homomorphism} between $K_1$ and $K_2$ is a field homomorphism $\sigma: K_1 \longrightarrow K_2$ such that $\sigma(r) = r$, for every $r\in\Q$. It is well known that there exist exactly $d$ distinct $\Q$-homomorphisms between a number field $K$ of degree $d$ and $\C$. We will call these $\Q$-homomorphisms between $K$ and $\C$ the \textit{embeddings} of $K$ into $\C$. If the image $\sigma(K)$ is contained in $\R$, we will say $\sigma$ is a \textit{real} embedding. Note that if $\sigma$ is not real, then $\overline{\sigma}$ is another embedding, and  we will call $\{\sigma,\overline{\sigma}\}$ a \textit{conjugate pair}.
	
	Let $\alpha_1$ be an algebraic number and $\alpha_2,\dots,\alpha_d$ the rest of the roots of its minimal polynomial. In the particular case that $K = \Q(\alpha_1)$, the $d$ distinct embeddings are completely determined by $\sigma_j(\alpha_1) = \alpha_j$, $j=1,\dots,d$. For $\eta\in K$, we call $\sigma_1(\eta),\dots,\sigma_d(\eta)$ the \textit{conjugates} of $\eta$. We define the \textit{trace} of $\eta\in K$ by $\textup{Tr}(\eta) = \sum_\sigma \sigma(\eta)$, where the sum runs over all embedding of $K$ in $\C$. In particular, if $\eta\in\mathcal{O}_K$ then $\textup{Tr}(\eta)\in\Z$.

\subsection{Proof of Theorem~\ref{th:noQWM}}
This is a minor variation of the proof of \cite[Theorem 2.7]{BuMMSo25}, for which we need a version of \cite[Lemma 6.2]{BuMMSo25}.
Thus we introduce notation consistent with \cite{BuMMSo25} rather than with  the earlier sections here.

Let $A\in GL(d,\Z)$ be an integer matrix, irreducible over $\Q$, whose characteristic polynomial is the minimal polynomial of
 a Salem number $\alpha$ of degree $d=2m+2$. Denote by $\alpha_0=\alpha^{-1}$,
  $\alpha_1,\alpha_2 \dots \alpha_{m}$, the real and complex conjugates of $\alpha$ in the (closed) upper half-plane, and let  $\sigma_0, \sigma_1,\sigma_2,\dots,\sigma_{m}$ be the corresponding embeddings.

Let $\{e_i\}^d_{i=1}$ be an eigenbasis of $A$ with entries in $\Z[\alpha]^d$. Consider a non-zero vector $\vec{s}\in\R^d$ such that
\[
\vec{s} = e_1 + C_2e_2 + \dots + C_de_d,
\]
that is, $C_1=1$. It is sufficient to consider such roof vectors for the suspension flow, since multiplying $\vec s$ by a (positive)
constant results in a
measurably isomorphic $\R$-action.
\begin{prop} \label{prop:noQWM}
	For every $\eps>0$ there exists $\eta\in\Z[\alpha]$ such that 
	\begin{equation} \label{dist1}
	\norm{\eta A^n\vec{s}}_{\R^d/\Z^d} < \eps, \quad \forall n\geq0.
	\end{equation}
\end{prop}
\begin{proof}
	The argument is  very similar to that  of \cite[Lemma 6.2]{BuMMSo25}. Denote $\omega^v = \omega \langle e_1,v\rangle$, for $v\in\Z^d$. Then,
	\begin{align}
		\langle \omega A^n\vec{s}, v\rangle &= \Tr(\omega^v\alpha^n) - \sigma_0(\omega^v)\alpha^{-n} - \sum_{j=1}^{m} \bigl(\sigma_j(\omega^v)\alpha^n_j + \overline{\sigma_j(\omega^v)}\alpha^{-n}_j\bigr) \nonumber \\
		& \quad +  C_2\omega^v\alpha^{-n} + \omega^v\sum_{j=1}^{m} \bigl(C_{2j+1}\alpha^n_j + \overline{C_{2j+1}}\alpha^{-n}_j\bigr) \nonumber \\
		&= \Tr(\omega^v\alpha^n) + (C_2\omega^v - \sigma_0(\omega^v))\alpha^{-n} + \sum_{j=1}^{m}  \bigl(H_j\alpha_j^{n} + \overline{H_j}\alpha_j^{-n}\bigr), \label{dist2}
	\end{align}
	where $H_j = C_{2j+1}\omega^v - \sigma_j(\omega^v)$. Consider the following embedding of $\Q(\alpha)$ into $\R^2\oplus \C^m \simeq \R^d$ (we identify $\C$ with $\R\oplus\R$):
	\begin{align*}
		\tau:\omega &\mapsto \begin{pmatrix}
			\om,&
			C_2\om - \sigma_0(\om),&
			C_3\om - \sigma_1(\om),&
			C_5\om - \sigma_2(\om),&
			\dots,&
			C_{2m+1}\om - \sigma_{m}(\om)
		\end{pmatrix} \\
		&\simeq \begin{pmatrix}
			\om,&
			C_2\om - \sigma_0(\om),&
			\re(C_3\om - \sigma_1(\om)),&
			\im(\om - \sigma_1(\om)),&
			\dots,&
			\im(C_{2m+1}\om - \sigma_{m}(\om))
		\end{pmatrix}
	\end{align*}
	The image of $\Z[\alpha]$ under this map is a lattice in $\R^d$. Indeed, we just need to verify that for a basis of $\Z[\alpha]$, its image is a basis for $\R^d$. We may use the power basis given by $\{1,\alpha,\dots,\alpha^{d-1}\}$. It suffices to show that the determinant of the next square real matrix of dimension $d$ is different from zero:
	\[\Small{
		M = \begin{pmatrix}
			1 & C_2-1& \re(C_3-1) & \im(C_3-1) & \cdots & \im(C_{2m+1}-1) \\
			\alpha & C_2\alpha - \alpha^{-1} & \re(C_3\alpha - \alpha_1) & \im(C_3\alpha - \alpha_1) & \cdots & \im(C_{2m+1}\alpha - \alpha_{m}) \\
			\vdots & \vdots & \vdots & \vdots & \vdots & \vdots \\
			\alpha^{d-1} & C_2\alpha^{d-1} - \alpha^{1-d} & \re(C_3\alpha^{d-1} - \alpha_1^{d-1}) & \im(C_3\alpha^{d-1} - \alpha_1^{d-1}) & \cdots & \im(C_{2m+1}\alpha^{d-1} - \alpha_{m}^{d-1})
	\end{pmatrix}}
	\]
	(this is just the matrix $(\tau(1),\tau(\alpha),\dots,\tau(\alpha^{d-1}))^{\sf T}$) written explicitly). By elementary operations on the columns, 
	which we can perform as vectors in $\C^d$, since it does not affect
	 whether the determinant vanishes or not, we obtain the complex matrix
	\[
	\widetilde{M} = \begin{pmatrix}
		1 & 1 & 1 & 1 & \cdots & 1 \\
		\alpha & \alpha^{-1} & \alpha_3 &  \alpha_4 & \cdots & \alpha_d \\
		\vdots & \vdots & \vdots & \vdots & \vdots & \vdots \\
		\alpha^{d-1} &  \alpha^{1-d} & \alpha_1^{d-1} &  \overline{\alpha_1}^{d-1} & \cdots & \overline{\alpha_{m}}^{d-1}
	\end{pmatrix},
	\]
	whose determinant is not equal to zero, since it is the Vandermonde matrix of different numbers.
	
	For any $\eps>0$, let $E>0$, to be specified later. Consider the system of inequalities in the variables $n_i\in \Z$, where 
	$\tau_j(\omega) = C_{2j+1}\omega - \sigma_j(\omega)$, $j=1,\dots,m$, and $\tau_0(\omega) = C_2\omega - \sigma_0(\omega)$:
	\[
	\left|\sum_{i=0}^{d-1} n_i \alpha^i\right| < E, \quad \left| \sum_{i=0}^{d-1} n_i \tau_j(\alpha^i) \right| < \eps, \quad \text{ for } j=0,1,\dots,m.
	\]
	By ``Minkowski's first theorem," (see \cite[Theorems B.1 and B.2]{Bug04}),  if
	\beq \label{eq:Mink}
	E>\const\cdot |\det(M)|\cdot \eps^{-(d-1)},
	\eeq
	then 
	there exists a non-trivial integer solution $(n_i)_{0\leq i\leq d-1}$ to the latter system. For such a solution, set $\eta = \sum_{i=0}^{d-1} n_i \alpha^i$.
	It satisfies \eqref{dist1}, in view of \eqref{dist2} and the observation that
	$$
	\norm{w}_{\R^d/\Z^d} \asymp \max_j |\langle w,{\mathbf e_j}\rangle|,\ \ w\in \R^d,
	$$
	where $\{{\mathbf e_j}\}_{j=1}^d$ is the standard basis in $\R^d$.
\end{proof}

The rest of the proof of Theorem~\ref{th:noQWM} proceeds exactly as in \cite[Section 6]{BuMMSo25}, the only difference being that the vector ${\bf 1}$ is replaced by a general vector $\vec{s}$. 
It is based on an upper bound for the lower local dimension of the spectral measure for a typical Lip-cylindrical function $f$, which uses
\cite[Lemma  6.3]{BuMMSo25}, and \cite[Theorem 4.3 and Corollary 4.4]{BuSo20}, yielding that for any $\eps>0$ and $\eta = \eta(\eps)$ from
Proposition~\ref{prop:noQWM} holds
\beq
\liminf \frac{\log\sig_f(B_r(\eta(\eps))}{\log r} < \const\cdot \eps,
\eeq
precluding H\"older QWM. The reader is referred to \cite{BuMMSo25} for details. \qed

\appendix

\section{$\log^*$ quantitative weak mixing for the self-similar suspension flow}

Define functions $x\mapsto \log^*(x),\ x\ge 1$  and $x\mapsto \log^*_{L}(x),\ x\ge L$ for $L>2$. Let $\log^k(x)$ be the $k$-th iterate of $\log$ (to base
	$2$), if it is well-defined. Similarly, let $\log_L^k(x)$  be the $k$-th iterate of $\log_L$ (log to base $L$). For $x\ge 1$ we let
	$$
	\log^*(x) := \min\{n\ge 0: \ \log^n(x)\le 1\},\ \ \ \log_L^*(x) :=\min\{n\ge 0: \ \log_L^n(x)\le 1\}.
	$$
	We start with an elementary lemma.
	
	\begin{lemma}
	Let $L>2$. Let $t_L>0$ and $K_L>0$ be defined as follows:
	$$
	\frac{\log t_L}{\log\log t_L} = \log L,\ \ \ \ \ \ K_L = \max\{4, t_L\}.
	$$
	Then
	\beq \label{ineq-star}
	\log_L^*(x) \ge \lfloor \shalf(\log^*(x) - \log^*(K_L))\rfloor.
	\eeq
	\end{lemma}
	
	\begin{proof}
	Using that $\frac{\log x}{\log\log x}$ is increasing and positive for $x\ge 4$, we obtain for $x\ge K_L$:
	$$
	\log_L x = \frac{\log x}{\log L} \ge \log\log x,\ \ \mbox{if}\ \ \frac{\log x}{\log\log x} \ge \log L
	$$
	Iterate this and obtain
	$$
	\log_L^m(x) \ge \log^{2m}(x),\ \ \mbox{where}\ \ m = \lfloor n/2\rfloor,\ \ n = \min\{k: \log^k(x) < K_L\}.
	$$
	This implies \eqref{ineq-star}.
	\end{proof}
	
\begin{theorem} \label{th:app1}
Let $\zeta$ be a substitution of Salem type on $d$ symbols, having a Salem number $\alpha$ of degree $d$ as the PF eigenvalue of the substitution matrix
$\Sf_\zeta$. Let $\vec s$ be the PF eigenvector of $\Sf_\zeta^T$, and let $(\Xxi, \phi_t,\wt\mu_{\vec s})$ be the corresponding suspension flow (the 
self-similar flow). Then for every Lip-cylindrical function $f$ of mean zero, the spectral measure
$\sig_f$ satisfies
\beq\label{eq:ssmain}
\sig_f(B_r(\om)) \le a_\alpha \exp\bigl[-c_\alpha \log^*(r^{-1})]\ \ \ \mbox{for all}\ \ r>0 \ \ \mbox{and}\ \ \om\in \R.
\eeq
Moreover, this system has QWM with the rate $\exp\bigl[-c_\alpha \log^*(r^{-1})]$  for such functions $f$.
\end{theorem}	

Let $\alpha_1=\alpha$, $\alpha_2,\ldots, \alpha_d$, be the eigenvalues of $\Sf_\zeta$, i.e., the Salem number $\alpha$ and its conjugates, so that 
$\alpha_2 = \alpha_1^{-1}$ and
$|\alpha_2|=1$ for $j=3,\ldots,d$.
The main part of the proof is contained in the following proposition.

\begin{prop} \label{prop:SalemQWM}
There exists an integer $L = L(\alpha) \ge 3$ and $\delta_1>0$, such that 
\beq \label{Salem-decay2}
\sum_{n=0}^{N-1}\|w\alpha^n\|_{_{\R/\Z}}^2 \ge \ell\delta_1,\ \ \mbox{for all}\ \ w\in [1,\alpha),
\eeq
where $N=k_\ell$ is defined inductively: $k_0=1,\ k_i=L^{k_{i-1}}$,\ $i\ge 1$, so that $\log_L^*(N) = \ell$.
\end{prop}

\begin{proof}[Proof of the proposition]
The proof is a minor modification of the argument in \cite[Proposition 5.5 and Corollary A.3]{BuSo14}. Therefore, we only give a sketch, referring to \cite{BuSo14}
for details. Fix $w\in [1,\alpha]$ and write
$$
w\alpha^n = p_n + \eps_n, \ \ \mbox{where}\ \  p_n \in \N\ \ \mbox{and}\ \  \eps_n\in [-1/2,1/2),\ n\ge 0,
$$
so that $|\eps_n| = \|w\alpha^n\|_{_{\R/\Z}}$, i.e., $p_n$ is the nearest integer to $w\alpha^n$. Note that $p_n\ge 1$ by the assumption on $w$.
Let $A$ be the companion matrix of the minimal
polynomial of $\alpha$, so that the vectors
$$
\bbe_j = (1, \alpha_j,\ldots, \alpha_j^{d-1})^{\sf T},\ \ j=1,\ldots, d,
$$
form a basis of eigenvectors for $A$ in $\C^d$. Consider two sequences of vectors:
$$
\vec p_k = (p_k,\ldots, p_{k+d-1})^{\sf T}\in \N^d\ \ \ \ \mbox{and}\ \ \ \ \vec \eps_k = (\eps_k,\ldots, \eps_{k+d-1})^{\sf T}\in \C^d,\ \ k\ge 0.
$$
Let
$
\delta_1:= (1 +  dH)^{-1},
$
where $d$ is the degree of $\alpha$ and $H$ is the height of its minimal polynomial, as in  \cite[(5.5)]{BuSo14}. It is easy to see that
\beq \label{eq:star1}
\max\{\|\vec\eps_k\|_\infty,\|\vec\eps_{k+1}\|_\infty\} < \delta_1 \implies \vec\eps_{k+1} = A \vec\eps_k,\ \ k\ge 0.
\eeq
Let
$$
\vec \eps_k = \sum_{j=1}^d a_j^{(k)} \vec \bbe_j
$$
be the decomposition of $\vec\eps_k$ into a linear combination of the eigenvectors. Observe that $\vec\eps_k = t\alpha^k \bbe_1 - \vec p_k$,
which implies that the coefficients $a_j^{(k)}$ do not vanish, and they can be estimated by modulus from below, using Garsia's Lemma
\cite[Lemma 1.51]{Garsia}, see \cite{BuSo14} for details. In particular,
\beq \label{eq:star2}
|a_3^{(k)}| \ge c_1 \alpha^{-kd}
\eeq
for some $c_1 = c_1(\alpha)$. Note that if $\vec\eps_{k+1} = A \vec\eps_k$, then $|a_3^{(k+1)}| = |a_3^{(k)}|$, since $|\alpha_3|=1$.
Together with \eqref{eq:star1} and \eqref{eq:star2}, this implies
$$
\max\{\|\vec{\eps}_k\|_\infty, \|\vec{\eps}_{k+1}\|_\infty,\ldots,\|\vec{\eps}_{k+n}\|_\infty\} <\delta_1\ \Longrightarrow\ \|\vec{\eps}_{k+n}\|\ge 
c_2 \alpha^{-kd},\ \ \mbox{for all}\ n\ge 1,
$$
with some constant $c_2 = c_2(\theta)$.
By the definition of $\vec\eps_k$ we obtain for all $k\ge 0$ and $n\ge 1$:
\beq \label{est1}
\max\left\{\|w\alpha^{k}\|_{_{\R/Z}},\ldots, \|w\alpha^{k+nd-1}\|_{_{\R/Z}}\right\} < \delta_1 \Longrightarrow\ \ \sum_{j=k}^{k+nd-1}
\|w \alpha^j\|^2_{_{\R/Z}} \ge c_2^2\,  n \alpha^{-2kd}.
\eeq
It will be convenient to break up the sum in \eqref{Salem-decay2} as follows:
\beq \label{split}
\sum_{n=0}^{N-1}\|w\alpha^n\|_{_{\R/\Z}}^2 = \sum_{i=0}^{\ell-1} \sum_{j=k_i}^{k_{i+1}-1} \|w\alpha^j\|_{_{\R/\Z}}^2,\ \ \mbox{where}\ \ 
k_0 = 1,\ k_{i} = L^{k_{i-1}},\ i\ge 1,
\eeq
so that $k_{\ell} = N$, as desired. We can choose an integer $L\ge 2,\ L\sim \alpha^{2d},$ so that
$$
L^k -k \ge \lceil (\delta_1/c_2^2)\cdot \alpha^{2kd}\rceil\cdot d,\ \ \mbox{for all}\ \ k\ge 1.
$$
Then, in view of \eqref{est1}, with
$n= \lceil (\delta_1/c_2^2)\cdot \alpha^{2kd}\rceil$, we  have
$$
\sum_{j=k}^{L^k-1} \|w\alpha^j\|_{_{\R/\Z}}^2 \ge \delta_1,\ \ \mbox{for all}\ k\ge 1.
$$
Finally, \eqref{split} implies $\sum_{n=0}^{N-1}\|w\alpha^n\|_{_{\R/\Z}}^2\ge \ell\delta_1$, proving \eqref{Salem-decay2}, as desired.
\end{proof}

\begin{proof}[Proof of Theorem~\ref{th:app1}]
Fix an arbitrary good return word $v$ for $\zeta$, possibly passing to a power of the substitution. By Proposition~\ref{prop:prod}, using that 
$\vec s$ is the PF eigenvalue of $\Sf_\zeta^{\sf T}$, we obtain for $f\in \Lip(\Xxi)$,\ $\om \in \R$, and $R\ge \alpha$:
\beq \label{eq-star3}
		\left|S^f_R((x,t),\omega)\right| \leq C_1\|f\|_\infty  \cdot \min\{1, |\om|^{-1}\} \cdot R\prod_{n=0}^{\floor{\log_{\alpha}R}} \left( 1-\lambda 
		\big\| \om\bigl\langle \vec s, \ell(v)\bigr\rangle\alpha^n\bigr\|^2_{_{\R/\Z}}\right)
\eeq
Assume first that $|\om| \in [b,b\alpha]$, where $b = \bigl\langle \vec s, \ell(v)\bigr\rangle^{-1}$. Then Proposition~\ref{prop:SalemQWM} implies
\begin{eqnarray*}
\left|S^f_R((x,t),\omega)\right| & \leq & C_1\|f\|_\infty  \cdot \min\{1, |\om|^{-1}\} \cdot R
\cdot \exp\bigl[-\lam\delta_1\bigl(\log^*_L(\log_\alpha R)-1\bigr)\bigr] \\
& \leq & C_1\|f\|_\infty  \cdot \min\{1, |\om|^{-1}\} \cdot R
\cdot \wt a_\zeta \exp\bigl[-\wt c_\zeta \log^*R \bigr], \ \ R\ge \alpha,
\end{eqnarray*}
for some constants $\wt a_\zeta$ and $\wt c_\zeta$, in view of \eqref{ineq-star}. Now suppose that $|\om| \ne 0$ and let
$K_\om = \lfloor \log_\alpha(|\om|/b)\rfloor\in \Z$, so that $|\om| \in [b \alpha^{K_\om}, b \alpha^{K_\om+1})$. Then we obtain from 
Proposition~\ref{prop:SalemQWM} and \eqref{eq-star3}, similarly to the above that
$$
\left|S^f_R((x,t),\omega)\right| \leq  C_1\|f\|_\infty  \cdot \min\{1, |\om|^{-1}\} \cdot R
\cdot \wt a_\zeta \exp\bigl[-\wt c_\zeta \log^*\bigl(R/\alpha^{|K_\om|}\bigr) \bigr], \ \ R\ge R_0(\om)=\alpha^{|K_\om|+1},
$$
This allows us to perform ``glueing near zero,'' similarly to Section~\ref{sec:glue}, for $f\in \Lip(X_\zeta^{\vec s})$ of mean zero, and obtain estimates
uniform in $\om$, which yield \eqref{eq:ssmain} via Lemma~\ref{lem-easy1}, and hence QWM with the $\log^*$ rate. The details are left to the
reader.
\end{proof}

	\bigskip

\section{On the Fourier decay for Salem Bernoulli convolutions} 

\subsection{Upper bound}
For $\lam\in (0,1)$ and $p\in (0,1)$ let 
$\nu_\lam^p$ be the distribution of the random series $\sum_{n=0}^\infty \pm \lam^n$, where the signs are chosen i.i.d., with probabilities $p,1-p$.
The ``classical'' Bernoulli convolutions are obtained for $p=\half$; the case when $p\ne \half$ is called ``biased''. This is a remarkable family of
``fractal'' measures, which have been extensively studied for almost 100 years, but many open questions remain. 
It is not our intention to give a general survey of the topic; the reader 
is referred to \cite{sixty,Varju} and references therein. However, we do recall the principal results on Fourier decay. It is easily seen that
\beq \label{Ber-Fourier}
\widehat{\nu}_\lam^p(\xi) =\prod_{n=0}^\infty \Bigl( p e^{-2\pi i \lam^n \xi} + (1-p) e^{2\pi i \lam^n \xi}\Bigr),\ \ \ \ \ 
\nulahat^{1/2}(\xi) = \prod_{n=0}^\infty \cos(2\pi \lam^n \xi). 
\eeq
Erd\H{o}s \cite{Erd1} discovered  that if $\lam^{-1}\ne \half$ is a Pisot number, then 
$\nulahat^p(\xi)$ does not vanish at infinity. Later, Salem \cite{Salem} proved
that in all other cases $\lim_{|\xi|\to \infty} \nulahat^p(\xi) = 0$. However, the question of decay rate is much more subtle and is still far from being
resolved. Erd\H{o}s \cite{Erd2} demonstrated that a power Fourier decay holds for Lebesgue-a.e.\ $\lam$, and Kahane \cite{Kahane} showed that
Erd\H{o}s' argument actually yields power Fourier decay for all $\lam$ outside of a set of Hausdorff dimension zero. Their argument forms a basis for what
became known as the ``Erd\H{o}s-Kahane technique''. For specific $\lam$ our knowledge is rather poor. Kershner \cite{Kersh} obtained a power of the
logarithm upper bound for Bernoulli convolutions corresponding to rational $\lam \ne 1/q$. In \cite{BuSo14} such bounds were shown for 
non-Pisot and non-Salem algebraic integers. For other results on Fourier decay of self-similar measures, see \cite{VarYu,Sahlsten} and references therein. To our knowledge, in the literature there are no quantitative upper bounds for the Fourier decay of Salem Bernoulli convolutions, i.e., when $\lam = \alpha^{-1}$ is a reciprocal of a Salem number. We show that (admittedly, ``super-weak'') $\log^*$ upper bound follows from Proposition~\ref{prop:SalemQWM}.

	\begin{corollary} \label{cor:Salem}
	Let $\alpha$ be a Salem number and $\lam = \alpha^{-1}$. For $p\in (0,1)$ consider the biased Bernoulli convolution $\nu_\lam^p$. Then
	there exist $A=A_{\alpha,p}$ and $C=C_{\alpha,p}$, such that 
		\beq \label{Salem-decay}
	|\nulahat^p(\xi)| \le A\exp\bigl(-C\log^*|\xi| \bigr),\ \ |\xi|\ge 1.
	\eeq
	\end{corollary}
		
	\begin{proof}[Sketch of the proof]
	We recall the (standard) proof of \cite[Corollary A.3]{BuSo14}.
	Let $\xi \in [\alpha^{N-1},\alpha^{N}]$, and set $t = \alpha^{-N}\xi$. Then 
\begin{eqnarray*}
|\widehat{\nu}_\lam^p(\xi) |  = \prod_{n=0}^\infty \Bigl| p e^{-2\pi i \lam^n \xi} + (1-p) e^{2\pi i \lam^n \xi}\Bigr| 
                                                  \le  \prod_{n=0}^{N-1} \Bigl| p +(1-p)e^{4\pi i \alpha^n t}\Bigr| 
                                                  \le  \prod_{n=0}^{N-1} \Bigl( 1 - {\frac{1-p}{2}}\|2\alpha^n t\|_{_{\R/\Z}}^2\Bigr),
\end{eqnarray*}
using an elementary inequality $|1 + e^{2\pi i \tau} | \le 2 - \half \|\tau\|_{_{\R/\Z}}^2$ for $\tau\in \R$. Now
Proposition~\ref{prop:SalemQWM} easily implies the desired claim.
\end{proof}

\subsection{Lower bounds} For completeness, we briefly discuss lower bounds for $|\nulahat^p(\xi)|$. They were suggested to us (in general terms, without
details) by Peter Varj\'u [personal communication]. These bounds are based on an observation of Kahane \cite{Kahane} 
(see \cite[Section 5]{sixty} for a detailed exposition)
that Salem
Bernoulli convolutions can never have a power Fourier decay. 
They follow from the well-known result that for a Salem number $\alpha$ and $\eps>0$
there exists a nonzero $\eta\in \Z[\alpha]$, such that $\|\eta \alpha^n\|_{\R/\Z}<\eps$  for all $n\in \N$ (of which Proposition~\ref{prop:noQWM} is
a generalization), with a quantitative bound similar to \eqref{eq:Mink}. We omit the details.

\begin{prop} \label{prop1}
Let $\alpha$ be a Salem number of degree $d$, $p\in (0,\half)$, and suppose that $\gam>0$ and $C\ge 1$ are such that for all $\xi\ge \alpha$ holds
\be \label{decay-cond}
\left|\nulahat^p(\xi)\right| \le C[\log_\alpha(\xi)]^{-\gam},\ \ \ \lam=\alpha^{-1}.
\ee
Then
\be \label{condi1}
\gam \le \shalf(d-1)\log_\alpha[(1-2p)^{-1}].
\ee
\end{prop}

For the unbiased Bernoulli convolution $\nula=\nula^{1/2}$ we got a weaker lower bound. The difference between these cases comes from the fact
that the terms of the product in \eqref{Ber-Fourier} are bounded in absolute value from below by $1-2p>0$, when $p\in (0,\half)$, whereas those for
$\nulahat^{1/2}$  vanish  at $\alpha^n/4$ and $3\alpha^n/4$.

\begin{prop} \label{prop2}
Let $\alpha$ be a Salem number of degree $d$, and suppose that $\beta,\gam>0$ and $C\ge 1$ are such that for all $\xi$ sufficiently large holds
\be \label{decay-cond2}
\left|\nulahat^{1/2}(\xi)\right| \le C[\log_\alpha(\xi)]^{-\gam(\log_\alpha\log_\alpha(\xi))^\beta},\ \ \ \lam=\alpha^{-1}.
\ee
Then 
$
\beta\le 1.
$
\end{prop}

\end{document}